\documentclass[12pt,reqno]{article}

\usepackage[usenames]{color}
\usepackage{amssymb}
\usepackage{amsmath}
\usepackage{amsthm}
\usepackage{amsfonts}
\usepackage{amscd}
\usepackage{graphicx}

\usepackage[colorlinks=true,
linkcolor=webgreen,
filecolor=webbrown,
citecolor=webgreen]{hyperref}

\definecolor{webgreen}{rgb}{0,.5,0}
\definecolor{webbrown}{rgb}{.6,0,0}

\usepackage{color}
\usepackage{fullpage}
\usepackage{float}

\usepackage{graphics}
\usepackage{latexsym}
\usepackage{epsf}

\usepackage{subcaption}
\newcommand{\mycaption}[2]{\caption[#1]{\, \textbf{#1.} #2}}

\newcommand{\seqnum}[1]{\href{https://oeis.org/#1}{\rm \underline{#1}}}

\begin{document}

	\theoremstyle{plain}
	\newtheorem{theorem}{Theorem}
	\newtheorem{corollary}[theorem]{Corollary}
	\newtheorem{lemma}[theorem]{Lemma}
	\newtheorem{proposition}[theorem]{Proposition}
	\theoremstyle{definition}
	\newtheorem{definition}[theorem]{Definition}
	\newtheorem{example}[theorem]{Example}
	\newtheorem{conjecture}[theorem]{Conjecture}
	\newtheorem{notation}[theorem]{Notation}
	\theoremstyle{remark}
	\newtheorem{remark}[theorem]{Remark}
	
	\newcommand{\floor}[1]{\left\lfloor#1\right\rfloor}
	\newcommand{\ceil}[1]{\left\lceil#1\right\rceil}

	\begin{center}
		\vskip 1cm{\LARGE\bf Takagi Function Identities\\
		\vskip .1in
		on Dyadic Rationals
		}
		\vskip 1cm
		\large
		Laura Monroe\footnote{This work was supported in part by the United States Department of Energy through the Los Alamos National Laboratory, operated by Triad National Security, LLC, for the National Nuclear Security Administration of the U.S. DOE under Contract No.\ 89233218CNA000001, and by the U.S. DOE ASC program at Los Alamos National Laboratory. }\\
		Ultrascale Systems Research Center\\
		Los Alamos National Laboratory\\
		Los Alamos, NM 87501\\
		USA\\
		\href{mailto:lmonroe@lanl.gov}{\tt lmonroe@lanl.gov} \\
	\end{center}
	
	\vskip .2 in
	\begin{abstract}
		The number of unbalanced interior nodes of divide-and-conquer trees on $n$ leaves is known to form a sequence of dilations of the Takagi function on dyadic rationals. We use this fact to derive identities on the Takagi function and on the Hamming weight of an integer in terms of the Takagi function.
	\end{abstract}
	
	\allowdisplaybreaks

\section{Introduction}
The Takagi function is a widely-studied self-similar continuous nowhere-differentiable function on $[0,1]$, identified by Takagi in 1901 \cite{takagi01}. It has connections to many areas of mathematics, including number theory, combinatorics, probability theory, and analysis, and many people have studied this function over the past century. Two recent comprehensive surveys on the subject have been published: one by  Lagarias \cite{lagarias11}, and one by Allaart and Kawamura \cite{allaart11}. The On-Line Encyclopedia of Integer Sequences (OEIS) sequence \seqnum{A268289} \cite{OEIS} is also related to the Takagi function. There are a number of identities linking this sequence with the Takagi function \cite{lagarias11,allaart11,baruchel19_2}.

There is an interesting connection between binary trees representing divide-and-conquer calculations and the Takagi function, shown by Coronado 
et al.\ \cite{coronado20}: the number of unbalanced interior nodes of a divide-and-conquer tree is a sequence of dilations of the Takagi function on dyadic rationals. 

In this paper we derive several formulas for counting such nodes on a divide-and-conquer tree, and use the dilations to extend them to identities on the Takagi function on dyadic rationals. These identities are in Section~\ref{dnodes2takagi}, and include the following:
\begin{itemize}
	\item Recursive and closed formulas for the Takagi function on dyadic rationals. The recursive formulas are in Theorems \ref{th_yet_another_takagi_prop} and \ref{takagi_another_recurrence}, and the closed are in Theorems  \ref{cor_takagi_exp_D_closed} and \ref{th_another_explicit_takagi}.
	\item A formula for the Hamming weight of an integer in terms of the Takagi function on dyadic rationals, in Theorem \ref{th_hamming_takagi}, derived from the formulas for counting $D$-nodes. As a consequence, we provide another proof of Trollope's theorem on cumulative binary digit sums \cite{trollope68} in Corollary~\ref{cor_trollope}, and give three succinct forms of the cumulative binary digit sum in Corollary~\ref{cor_trollope3way}.
\end{itemize}

\section{Notation}
We assume throughout the paper that $n>0$ is an integer, and use the following representations. If additional assumptions are made in individual theorems on these variables, we state them explicitly; otherwise, the definitions in use are as stated in Notation~\ref{not_n}.
\begin{notation}\label{not_n}
	A positive integer $n$ may be represented as:
	\begin{itemize}
		\item $n=2^k+r$, with integers $k$ and $r$ such that $k\ge 0$ and $0\le r<2^k$.
		\item $n=2^k\cdot (1+x)$, with an integer $k\ge 0$ and a dyadic rational $x$ with $0 \le x < 1$ .
		\item  $n = \sum_{i=0}^k n_i \cdot 2^i$, where $k=\floor{\log_2(n)}$ (binary decomposition). This may also be represented as $n=n_k n_{k-1} \cdots n_0$.
	\end{itemize}
\end{notation}
\begin{lemma}\label{lem_conversion}
	Let $n, k, r, x$ be as in Notation~\ref{not_n}. Then
	$x=\frac{r}{2^k}$ and $1+x = \frac{n}{2^k}.$
\end{lemma}
\section{The Takagi function and divide-and-conquer trees}
The Takagi function, shown in Figure~\ref{takagidilation}(\subref{takagifig}), is a continuous nowhere-differentiable function on $[0,1]$. 
We recast the original definition of the Takagi function on dyadic rationals from \cite{takagi01} to be consistent with notation and numbering convenient to the study of binary trees.
\begin{definition}[\textbf{Takagi function}] \label{takagi_newdef} \cite{takagi01}
	Let $n$ be an integer. The \emph{Takagi function} on dyadic rationals is defined as
	$$
	\tau\left(\frac{n}{2^{k+1}}\right) 
	= \sum_{i=1}^\infty\frac{\ell_i(n)}{2^i} 
	\text{, where } 
	\ell_{i+1}(n)=
	\begin{cases}
		\sum\limits_{j=0}^{i-1} n_{k-j}, & \text{ if } n_{k-i} = 0;\\[1.25em]
		i-\sum\limits_{j=0}^{i-1} n_{k-j}, & \text{ if } n_{k-i} = 1,\\
	\end{cases}
	$$
	where $n_i=0$ for every negative integer $i$.
\end{definition}
Theorem \ref{takagi2} gives a simple formula for the dilations of the Takagi function on dyadic rationals in [0,1], in terms of certain interior nodes of full binary trees, called $D$-nodes. 
\begin{definition}[\textbf{$S$-nodes and $D$-nodes}] \cite{monroe21}
	An interior node of a rooted full binary tree is called an \emph{$S$-node} if its two subtrees have the \textbf{S}ame number of leaves, and a \emph{$D$-node} if its subtrees have \textbf{D}ifferent numbers of leaves. Such a labeling gives an \emph{$SD$-tree}. 
\end{definition}
The $SD$-tree was defined to express the notion of balance in a binary commutative non-associative product \cite{monroe21}.

The next lemma is true because an $n$-leaf binary tree has $n-1$ interior nodes \cite{knuth97art1}, and these are partitioned into $S$-and $D$-nodes.
\begin{lemma} 
	\label{dnodes}
	Let $a(n)$ be the number of $S$-nodes in a binary tree having $n$ leaves. Then the number of $D$-nodes is $n-1-a(n)$. 
\end{lemma} 
\begin{definition}[\textbf{Divide-and-conquer tree}]
	A \emph{divide-and-conquer tree} is  a full binary tree in which the number of leaf descendants of the left and right children of any node differ by at most $1$.
\end{definition} 

The binary tree generated is thus as balanced as it can be. At every level, the leaf descendants of an interior node are evenly or almost evenly divided between the left and the right branches. 
There is only one divide-and-conquer tree on $n$ leaves, up to tree isomorphism.
\begin{notation}
	The number of $D$-nodes in a divide-and-conquer tree is denoted by $\delta(n)$.
\end{notation}
Theorem \ref{takagi2} is the same as Corollary 4 of Coronado 
et al.\ \cite{coronado20}
but is stated here in terms of $D$-nodes and Takagi's original definition. It is also closely related to 
Baruchel's proposition 3.3 \cite{baruchel19_2}. We show several examples of these dilations in Figure \ref{takagidilation}.

\begin{figure}[htb]
	\centering
	\begin{subfigure}{0.2\textwidth} 
		\includegraphics[width=1\textwidth]{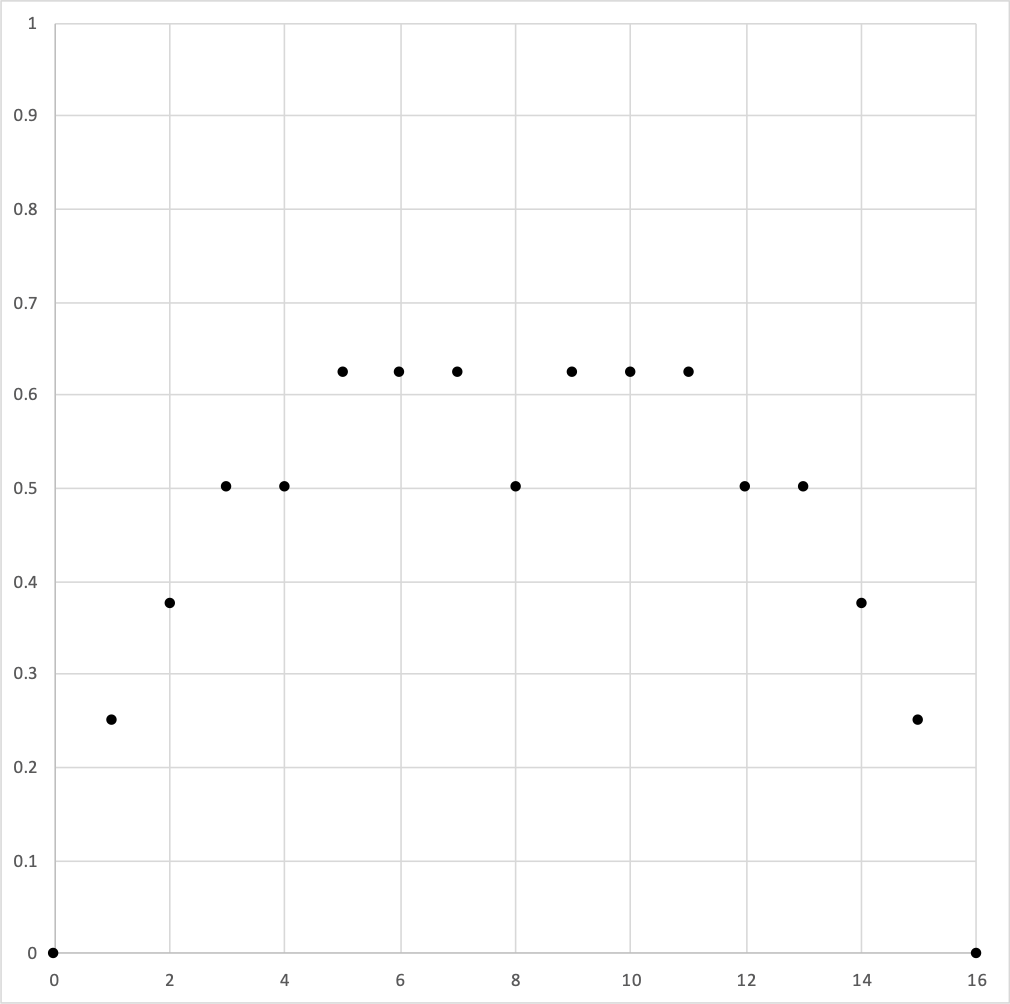}
		\caption{$y=\frac{\delta(16+x)}{16}$}
		\label{delta16}
	\end{subfigure}
	\hspace{1em}
	\begin{subfigure}{0.2\textwidth}
		\includegraphics[width=1\textwidth]{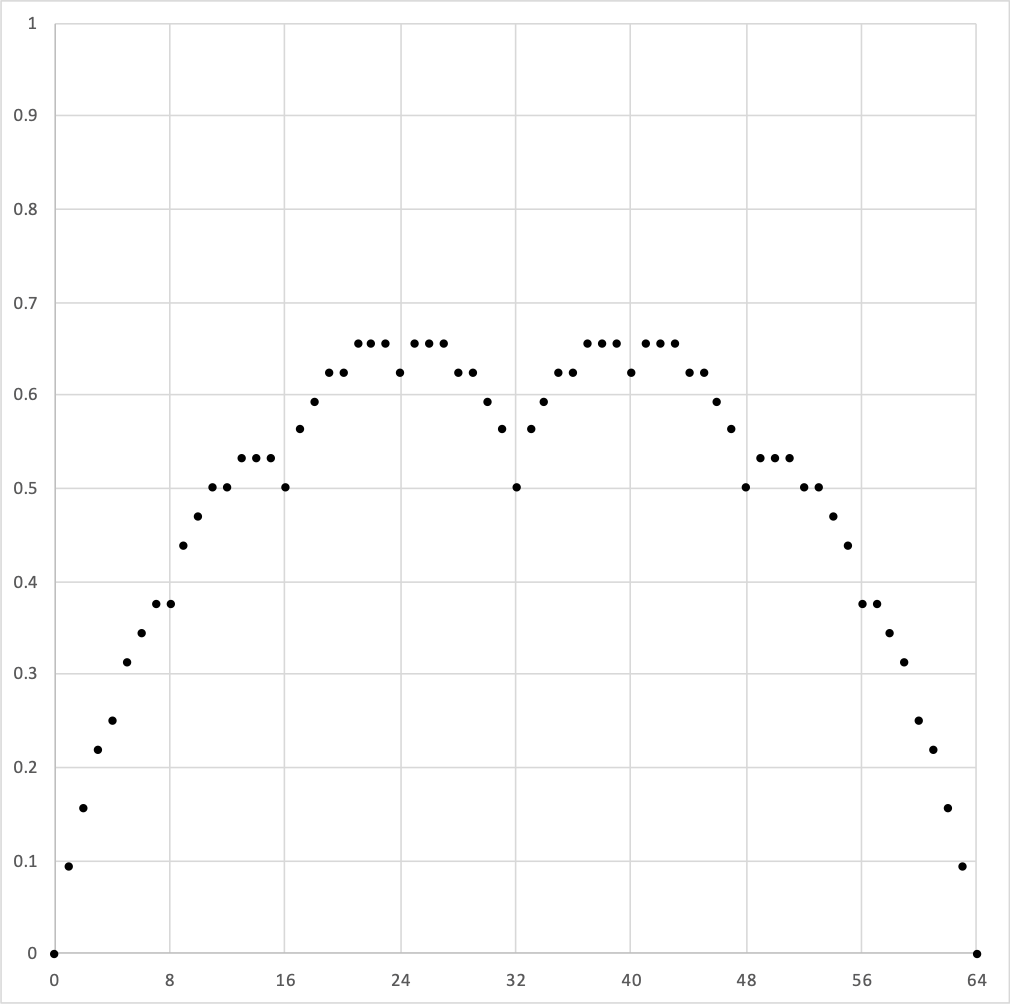}
		\caption{$y=\frac{\delta(64+x)}{64}$}
		\label{delta64}
	\end{subfigure}
	\hspace{1em}
	\begin{subfigure}{0.2\textwidth}
		\includegraphics[width=1\textwidth]{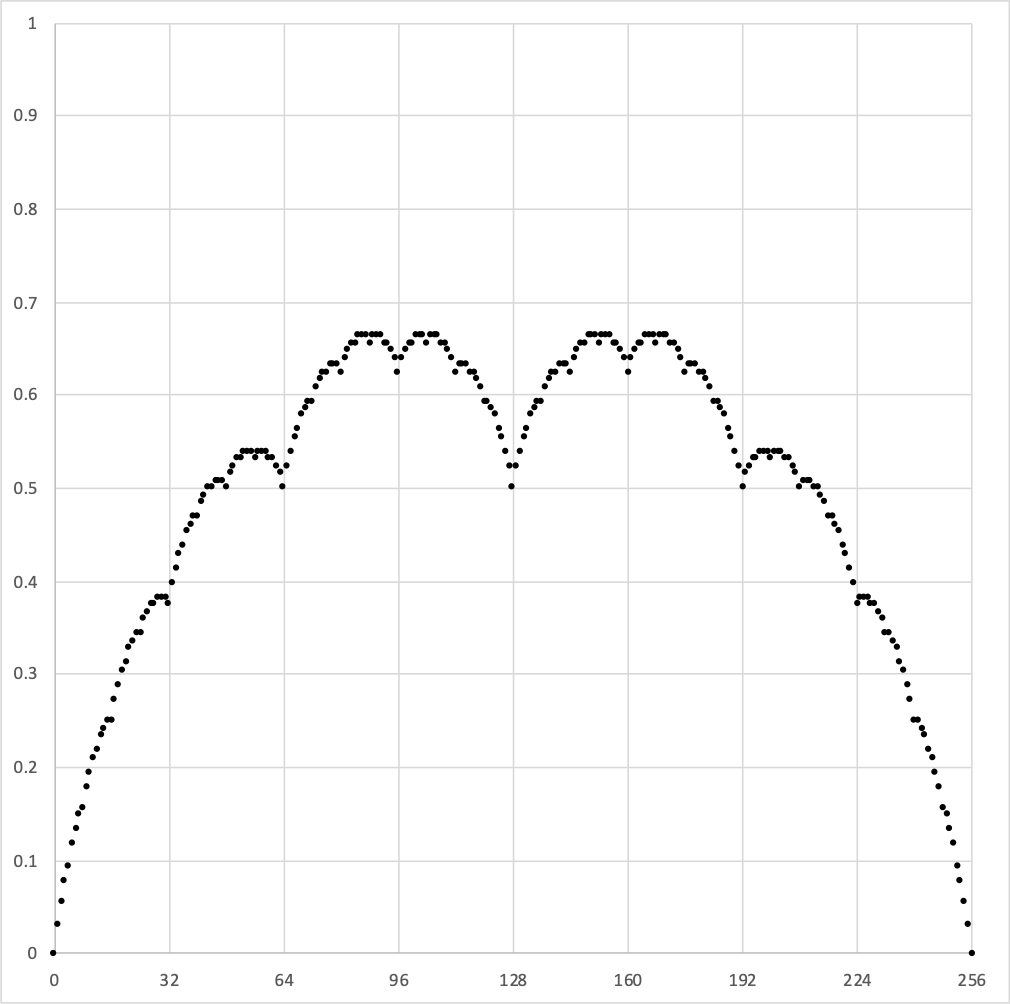}
		\caption{$y=\frac{\delta(256+x)}{256}$}
		\label{delta256}
	\end{subfigure}
	\hspace{1em}
	\begin{subfigure}{0.2\textwidth}
		\includegraphics[width=1\textwidth,height=1\textwidth]{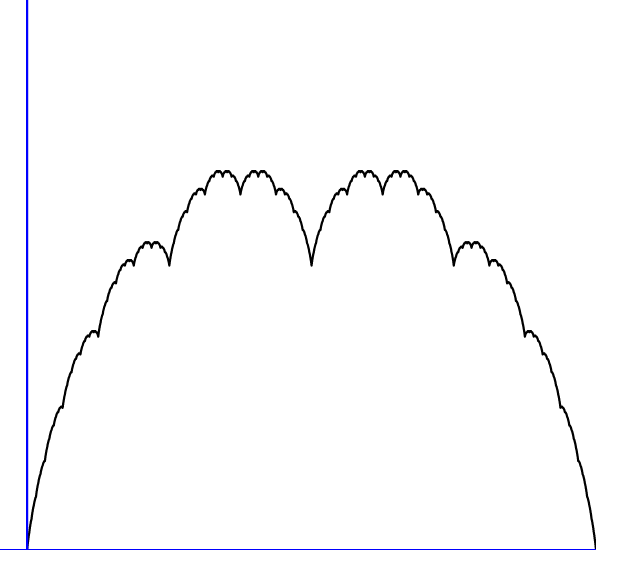}
		\caption{Takagi curve.}
		\label{takagifig}
	\end{subfigure}
	\mycaption{Divide-and-conquer dilations of the Takagi function on the dyadic rationals}
	{Subfigures (\subref{delta16}), (\subref{delta64}) and (\subref{delta256}) show examples of the dilations $y=\frac{\delta(2^k+x)}{2^k}=\tau\left(\frac{x}{2^k}\right)$  from Theorem \ref{takagi2}, where $\delta(n)$ is the number of $D$-nodes on a divide-and-conquer tree on $n$ leaves. Here, $k=4,6,\text{ and }8$, and $x$ is an integer with $0 \le x \le 2^k$. These may be visually compared to Subfigure (\subref{takagifig}), showing
		the continuous, self-similar, nowhere-differentiable Takagi (blancmange) curve on $[0,1]$. (The blancmange curve image  in Subfigure (\subref{takagifig}) is taken from Wiki Commons.)}
	\label{takagidilation}
\end{figure}
\begin{theorem} \label{takagi2} \cite{coronado20}
	Let $\delta(n)$ be the number of $D$-nodes in a divide-and-conquer tree on $n=2^k+r$ leaves, with $0 \le r \le 2^k$. Let $\tau(x)$ be the Takagi function on $x \text{, with }0 \le x \le 1$. Then 
	$$
	\tau \left( \frac{r}{2^k} \right) = \frac{\delta(2^k+r)}{2^k}.
	$$
\end{theorem}

\section{\texorpdfstring{$D$}--nodes in divide-and-conquer trees} \label{pairwise_cnaos}
In this section, we develop several identities on $\delta(n)$, the number of $D$-nodes in a divide-and-conquer tree. The first subsection is general, and the second develops identities in terms of $s_1(n)$, the Hamming weight of $n$. We will then apply Theorem~\ref{takagi2} to these in Section~\ref{dnodes2takagi}, for some quick identities on the Takagi function.
\subsection{Counting \texorpdfstring{$D$}--nodes in a divide-and-conquer tree}
We state here formulas for the number of $D$-nodes in a divide-and-conquer form with $n$ leaves. The number of $D$-nodes corresponds to OEIS sequence \seqnum{A296062}
\cite{OEIS}, as noted by Coronado et al.\ \cite{coronado20}.
\begin{theorem}\label{exp_D_recursive} \cite{coronado20}
	The number of $D$-nodes in a divide-and-conquer $SD$-tree with $n$ leaf nodes is
	$$
	\delta(n)=
	\begin{cases}
		2\delta(m), & \text{if}\ n=2m; \\
		\delta(m)+\delta(m+1)+1, & \text{if}\ n=2m+1;\\
		0, & \text{if}\ n=1.
	\end{cases}
	$$
\end{theorem}
The following corollary applies to $n$ when $n$ is not a power of $2$. 
\begin{corollary} 
	\label{yet_another_dnodes_prop}
	Let $n=2^k+r$, where $0<r<2^k$, and let $\rho_1$ be the position of the smallest non-$0$ bit in $r$. The number of $D$-nodes in a divide-and-conquer $SD$-tree with $n$ leaf nodes satisfies the equation
	$$
	\delta(n)=\frac{1}{2}(\delta(n-1)+\delta(n+1))+1-\rho_1.
	$$
\end{corollary}
\begin{proof}
	If $r=2d+1$ is odd, then $\rho_1=0$, and 
	\begin{align*}
		\delta(n) &= \delta(2(2^{k-1}+d)+1)\\
		&= \delta(2^{k-1}+d)+\delta(2^{k-1}+d+1)+1&\text{(by Theorem \ref{exp_D_recursive})}&\\
		&= \frac{1}{2}(2\cdot \delta(2^{k-1}+d)+2\cdot \delta(2^{k-1}+d+1)) +1\\
		&= \frac{1}{2}(\delta(2^k+2d)+\delta(2^k+2d+2)) +1 &\text{(by Theorem \ref{exp_D_recursive})}\\
		&= \frac{1}{2}(\delta(n-1)+\delta(n+1)) +1-0 \\
		&= \frac{1}{2}(\delta(n-1)+\delta(n+1)) +1 -\rho_1. 
	\end{align*}
	Now let $r=2d$ be even. First, we use what has been shown for $n$ odd:
	\begin{equation*} \label{1steq_69X}
		\delta(n-1) 
		= \frac{1}{2}(\delta(n-2)+\delta(n)) +1 
		= \frac{1}{2}\delta(n-2)+\frac{1}{2}\delta(n) +1, 
	\end{equation*}
	and
	\begin{equation*}\label{2ndeq_6X9}
		\delta(n+1) 
		= \frac{1}{2}(\delta(n)+\delta(n+2)) +1 
		= \frac{1}{2}\delta(n+2) +\frac{1}{2}\delta(n)+1.
	\end{equation*}
	Thus
	$$
	\delta(n-1)+\delta(n+1) 
	= \delta(n) +\frac{1}{2}(\delta(n-2)+\delta(n+2)) +2, 
	$$
	and
	\begin{equation} \label{n_plusminus_2}
		\delta(n-1)+\delta(n+1)-2-\delta(n) = \frac{1}{2}(\delta(n-2)+\delta(n+2)). 
	\end{equation}
	Now we apply induction on $k$, noting that the smallest non-$0$ bit in $\frac{n}{2}=(2^{k-1}+d)$ is ($\rho_1-1$):
	\begin{align*}
		\delta(n) &= \delta(2(2^{k-1}+d)) \\
		&= 2\delta(2^{k-1}+d)& \text{(by Theorem \ref{exp_D_recursive})}&\\
		&= 2\cdot (\frac{1}{2}(\delta(2^{k-1}+d-1)+\delta(2^{k-1}+d+1))+1-(\rho_1-1))& \text{(by induction)}\\
		&=  \frac{1}{2}(2\cdot\delta(2^{k-1}+d-1)+2\cdot\delta(2^{k-1}+d+1))+2-(2\cdot\rho_1-2) \\
		&=  \frac{1}{2}(\delta(2^k+2d-2)+\delta(2^k+2d+2))+2-(2\cdot\rho_1-2)& \text{(by Theorem \ref{exp_D_recursive})}\\
		&=  \frac{1}{2}(\delta(n-2)+\delta(n+2))+4-2\cdot\rho_1 \\
		&=  \delta(n-1)+\delta(n+1)-2-\delta(n)+4-2\cdot\rho_1& \text{(by Equation (\ref{n_plusminus_2}))}\\
		&=  \delta(n-1)+\delta(n+1)-\delta(n)+2-2\cdot\rho_1.
	\end{align*}
	So
	$$
	2\cdot \delta(n) 
	= \delta(n-1)+\delta(n+1)+2-2\cdot\rho_1,
	$$
	and
	$$
	\delta(n) 
	= \frac{1}{2}(\delta(n-1)+\delta(n+1))+1-\rho_1.
	$$
\end{proof}

\begin{theorem}
	\label{exp_D_closed}
	The number of $D$-nodes in a divide-and-conquer $SD$-tree with $n$ leaf nodes  is
	\begin{equation} \label{eq_exp_D_closed}
		\delta(n)=\sum_{i=0}^{\floor{\log_2(n)}-1} \lambda_i(n) \text{, where } \lambda_i(n)= 
		\begin{cases}
			(n\bmod 2^i), & \text{ if } n_i=0;\\
			2^i-(n\bmod 2^i), & \text{ if } n_i=1.\\
		\end{cases}	  
	\end{equation}
	An explicit form is
	\begin{equation} \label{eq_exp_D_closed_explicit}
		\delta(n)=\sum_{i=0}^{\floor{\log_2(n)}-1}  (n_i\cdot 2^i + (-1)^{n_i}\cdot (n \bmod 2^i)).
	\end{equation}
\end{theorem}
\begin{proof}
	Because every $SD$-tree is a full complete binary tree, there are $2^i$ nodes at level  $i$ except for the bottom level. At each level, the descendant leaf nodes are almost evenly divided between the nodes: at the $i^\text{th}$ level, $2^i-(n\bmod 2^i)$ nodes have $\floor{\frac{n}{2^i}}$ leaf descendants, and $(n\bmod 2^i)$ of these nodes have $\floor{\frac{n}{2^i}}+1$ leaf descendants. 
	
	Consider the binary representation of the number $n$ where the bits are ordered from least significant to most significant.  
	A node in a divide-and-conquer tree is an $D$-node if and only if it has an odd number of leaf descendants. 
	If $n_i$ is $0$, then $\floor{\frac{n}{2^i}}$ is even and $(n\bmod 2^i)$ nodes have an odd number of  descendants. If $n_i$ is $1$, then $\floor{\frac{n}{2^i}}$ is odd, and $2^i-(n\bmod 2^i)$ nodes have an odd number of  descendants. So the number of $D$-nodes at level $i$ is
	\begin{equation*}\label{eqA_DX}
		\beta_i(n) =
		\begin{cases}
			n\bmod 2^i, & \text{ if $n_i$ is }  0;\\
			2^i-(n\bmod 2^i), & \text{ if $n_i$ is }  1.\\
		\end{cases}	  
	\end{equation*}  
	Equation (\ref{eq_exp_D_closed}) is obtained by summing across the levels.
	
	So the number of $D$-nodes at level $i$ is
	$$
	n_i\cdot 2^i + (-1)^{n_i}\cdot (n \bmod 2^i).
	$$
	Equation (\ref{eq_exp_D_closed_explicit}) is obtained by summing across the levels.
\end{proof}

An extensive analysis of solutions to divide-and-conquer recurrences is provided in \cite{hwang17}. Theorem~\ref{exp_D_closed} might be obtained from Theorem \ref{exp_D_recursive} using the methods in that paper. Instead we prove Theorem \ref{exp_D_closed} directly by analyzing $D$-nodes.

\subsection{Counting \texorpdfstring{$D$}--nodes on a divide-and-conquer tree in terms of Hamming weight}
We begin this section with some simple lemmas on Hamming weight used throughout this section that appear in many papers. They can be seen by inspection.
\begin{definition}[\textbf{Hamming weight}]
	The \emph{Hamming weight} of $n$,  $\sum_{i=0}^k n_i$, is the number of $1$s in the binary expansion of $n$, and is denoted in this paper by $s_1(n)$.
\end{definition}
\begin{lemma}\label{weightlemma_addone} 
	The Hamming weight $s_1(n) = s_1(r)+1$.
\end{lemma}
\begin{lemma}\label{weightlemma_even} 
	If $n$ is even, then $s_1(n+1) = s_1(n)+1$.
\end{lemma}
\begin{lemma}\label{weightlemma_odd} 
	If $n$ is odd, then $s_1(n-1)=s_1(n)-1$. 
\end{lemma}
\begin{lemma} \label{weightlemma} 
	If $r<2^{k-1}$ is odd, 
	then $s_1(2^k-r)=1+k-s_1(r)$.
\end{lemma}
\begin{proof}
	The binary representation of $2^k$ is $1\ 0...0$, where there are $k$ $0$s following the leading $1$. Let the binary representation of $r$ be $0\ 0\ r_{k-2}\cdots\ r_1\ 1$, since $r<2^{k-1}$ is odd. Subtracting $r$ from $2^k$ gives $0\ 1\ \overline{r_{k-2}}\cdots\ \overline{r_1}\ 1$, where $\overline{r_i}$ is the negation of $r_i$. The lemma follows.
\end{proof}
We use the above lemmas to prove a series of assertions on the $D$-nodes of a tree with $n$ leaves.
\begin{lemma}
	\label{lambdadiff}
	Let $n=2^k+r$, with $r$ odd and $0 < r < 2^k$. Let $\lambda$ be as in Theorem~\ref{exp_D_closed}, with $0 < i<k$. Then
	$$
	\lambda_i(n)-\lambda_i(n-1)  =
	\begin{cases}
		1, & \text{ if }\ n_i=0;\\
		-1, & \text{ if }\ n_i=1.\\
	\end{cases}
	$$
\end{lemma}
\begin{proof}
	For $0<i<k$, all bits of $r$ and $(r-1)$ besides the $0^\text{th}$ are the same, since $r$ is odd. Thus
	$$
	\lambda_i(n)-\lambda_i(n-1)  =
	\begin{cases}
		(n \bmod 2^i)-((n-1) \bmod 2^i)=1, & \text{ if }\ n_i=0;\\
		(2^i-(n\bmod 2^i))-(2^i-((n-1)\bmod 2^i))=-1, & \text{ if }\ n_i=1.\\
	\end{cases}
	$$
\end{proof}
The next lemma follows from Theorem \ref{exp_D_recursive} by induction.
\begin{lemma}
	\label{symmetricity}
	Let $0 \le r \le 2^k$. Then $\delta(2^{k+1}-r) = \delta(2^k+r)$.
\end{lemma}
\begin{lemma}
	\label{dnodes_diff_even}
	Let $n>0$ be even. Then
	$$
	\delta(n+1) = \delta(n) + \floor{\log_2(n)}-2\cdot s_1(n)+2.
	$$
\end{lemma}
\begin{proof}
	Let $n=2^k+r$, with $0 \le r <  2^k$, and let $n$ 
	have binary expansion $n_k n_{k-1} \cdots n_0$. Now $n$ is even, so $\floor{\log_2(n)}=\floor{\log_2(n+1)}$.
	\begin{align*}
		\delta(n+1)-\delta(n)
		&= \sum_{i=0}^{k-1} \lambda_i(n+1)-\sum_{i=0}^{k-1} \lambda_i(n)& \text{(by Theorem \ref{exp_D_closed})}&\\
		&= \sum_{i=0}^{k-1} (\lambda_i(n+1)- \lambda_i(n))\\
		&= \lambda_0(n+1)- \lambda_0(n) + \sum_{i=1}^{k-1} (\lambda_i(n+1)- \lambda_i(n)).
	\end{align*}
	Since $n$ is even, $n_0=0$, so $\lambda_0(n+1)-\lambda_0(n)=(2^0-0)-0=1$.
	
	There are $s_1(n)-1=(s_1(n+1)-2)$ $1$s  in the $1^\text{st}$ through $(k-1)^\text{th}$ bits of $n+1$.
	Thus there are $((k-1)-(s_1(n+1)-2))$ $0$s in $1^\text{st}$ through $(k-1)^\text{th}$ bits of $(n+1)$. So 
	\begin{align*}
		\delta(n+1)-\delta(n)
		&= 1 -1\cdot (s_1(n+1)-2) + 1\cdot  ((k-1)-(s_1(n+1)-2))&\text{(by Lemma \ref{lambdadiff})}&\\
		&= k-2\cdot s_1(n+1)+4 \\
		&= k-2\cdot(s_1(n)+1)+4&\text{(by Lemma \ref{weightlemma_even})}&\\
		&= k-2\cdot s_1(n)+2\\
		&= \floor{\log_2(n)}-2\cdot s_1(n)+2.
	\end{align*}
\end{proof}
\begin{lemma}
	\label{dnodes_diff_odd}
	Let $n=2^k+r>0$ be odd. Then    $$
	\delta(n+1) = \delta(n)+\floor{\log_2(n)}-2\cdot s_1(n)+2.
	$$
\end{lemma}
\begin{proof}
	Let $n=2^k+r$, with $0 < r <  2^k$. We proceed by moving from $n$ and $n+1$ to the symmetric $m=2^{k+1}-(r+1)$ (which is even) and $m+1=2^{k+1}-r$ (which is odd), by Lemma \ref{symmetricity}. We then apply Lemma \ref{dnodes_diff_even}, and move back to $n$ and $n+1$, again by Lemma \ref{symmetricity}. 
	\begin{align*}
		\delta(n+1) 
		&= \delta(m) &\text{(by Lemma \ref{symmetricity})}&\\
		&= \delta(m+1)-(\floor{\log_2(m)}-2\cdot s_1(m)+2), &\text{(by Lemma \ref{dnodes_diff_even})}\\
		&= \delta(m+1)-(k-2\cdot s_1(m)+2) \\
		&= \delta(n)-(k-2\cdot s_1(m)+2) &\text{(by Lemma \ref{weightlemma})}\\ 
		&= \delta(n)-(k-2\cdot(s_1(m+1)-1)+2)&\text{(by Lemma \ref{weightlemma_even})}\\
		&= \delta(n)-(k-2\cdot s_1(2^{k+1}-r)+4)\\
		&= \delta(n)-(k-2\cdot(2+k-s_1(r))+4)&\text{(by Lemma \ref{weightlemma})}\\
		&= \delta(n)-(k-2\cdot(2+k-(s_1(n)-1))+4)&\text{(by Lemma \ref{weightlemma_addone})}\\
		&= \delta(n)+k-2\cdot s_1(n)+2\\
		&= \delta(n)+\floor{\log_2(n)}-2\cdot s_1(n)+2.
	\end{align*}
\end{proof}
We then have a general recurrence relation for $\delta(n)$, following from Lemmas \ref{dnodes_diff_even} and \ref{dnodes_diff_odd}.
\begin{theorem} 
	\label{dnodes_another_recurrence}
	The number of $D$-nodes, $\delta(n+1)$, in a divide-and-conquer tree with $n+1$ leaf nodes satisfies the equation
	$$
	\delta(n+1) 
	= \delta(n) + \floor{\log_2(n)}-2\cdot s_1(n)+2.
	$$
\end{theorem}
Repeated application of Theorem \ref{exp_D_recursive} (for even sub-sums) and Lemma \ref{dnodes_diff_even} (for odd) gives rise to Theorem \ref{another_explicit_D}, another explicit formula for the number of $D$-nodes in a divide-and-conquer tree.
\begin{theorem}\label{another_explicit_D}
	The number of $D$-nodes, $\delta(n)$, in a divide-and-conquer tree with $n+1$ leaf nodes satisfies the equation
	$$
	\delta(n) = \sum_{i=0}^{k-1}2^i \cdot n_i \cdot \bigl(k-i-2\cdot \sum_{j=i}^k n_i+4\bigr),
	$$
	or alternatively
	$$
	\delta(n) = \sum_{i=0}^{k-1}2^i \cdot n_i \cdot \bigl(k-i-2\cdot s_1\bigl(\floor{n/2^i}\bigr)+4\bigr).
	$$
\end{theorem}
\begin{proof}
	Let $n$ have binary expansion $n_k n_{k-1} \cdots n_0$. We proceed by induction, noting that the theorem is true for $n=1$ and $n=2$. 
	\newline
	If $n$ is even, the quotient $n/2$ has binary expansion $n_k n_{k-1} \cdots n_1$, so 
	\begin{align*}
		\delta(n/2)
		&=\sum_{i=1}^{k-1}2^{i-1} \cdot n_i \cdot \bigl(k-i-2\cdot \sum_{j=i}^k n_i+4\bigr)& \text{(by induction)}&\\
		&=\frac{1}{2}\sum_{i=1}^{k-1}2^i \cdot n_i \cdot \bigl(k-i-2\cdot \sum_{j=i}^k n_i+4\bigr),
	\end{align*}
	and
	\begin{align*}
		\delta(n)
		&=2\delta(n/2) &\text{(by Theorem \ref{exp_D_recursive})}&\\
		&=2\cdot \frac{1}{2}\sum_{i=1}^{k-1}2^i \cdot n_i \cdot \bigl(k-i-2\cdot \sum_{j=i}^k n_i+4\bigr) \\
		&=\sum_{i=0}^{k-1}2^i \cdot n_i \cdot \bigl(k-i-2\cdot \sum_{j=i}^k n_i+4\bigr).& \text{(since $n_0=0$)}&
	\end{align*}
	\newline
	If $n$ is odd then $n_0=1$ and $n-1$ has binary expansion $m_k m_{k-1} \cdots m_0$, where $m_i=n_i$ for $1\le i \le k$, and $m_0=0$. Then
	\begin{align*}
		\delta(n)
		&=\delta(n-1) + (\floor{\log_2(n-1)}-2\cdot s_1(n-1)+2)& \text{(by Theorem \ref{dnodes_another_recurrence})}&\\
		&=\delta(n-1) + (k-2\cdot s_1(n-1)+2) \\
		&=\sum_{i=0}^{k-1}2^i \cdot m_i \cdot \bigl(k-i-2\cdot \sum_{j=i}^k m_i+4\bigr) + (k-2\cdot s_1(n-1)+2)& \text{(by induction)}\\
		&=\sum_{i=1}^{k-1}2^i \cdot m_i \cdot \bigl(k-i-2\cdot \sum_{j=i}^k m_i+4\bigr) + (k-2\cdot s_1(n-1)+2)\\
		&=\sum_{i=1}^{k-1}2^i \cdot n_i \cdot \bigl(k-i-2\cdot \sum_{j=i}^k n_i+4\bigr) + (k-2\cdot s_1(n-1)+2)\\
		&=\sum_{i=1}^{k-1}2^i \cdot n_i \cdot \bigl(k-i-2\cdot \sum_{j=i}^k n_i+4\bigr) + (k-2\cdot(s_1(n)-1)+2)& \text{(by Lemma \ref{weightlemma_odd})}\\
		&=\sum_{i=1}^{k-1}2^i \cdot n_i \cdot \bigl(k-i-2\cdot \sum_{j=i}^k n_i+4\bigr) + \bigl(k-2\cdot \sum_{j=0}^k n_i+4\bigr)\\
		&=\sum_{i=1}^{k-1}2^i \cdot n_i \cdot \bigl(k-i-2\cdot \sum_{j=i}^k n_i+4\bigr) +  \bigl((k-0)-2\cdot \sum_{j=0}^k n_i+4\bigr) \\
		&=\sum_{i=0}^{k-1}2^i \cdot n_i \cdot \bigl(k-i-2\cdot \sum_{j=i}^k n_i+4\bigr).& \text{(since $2^0\cdot n_0=1$)}
	\end{align*}
\end{proof}
Theorem~\ref{dnodes_another_recurrence} also gives a corollary on the Hamming weight of $n$.
\begin{corollary}\label{cor_hamming_delta}
	The Hamming weight of $n$, $s_1(n)$, satisfies the equation
	$$
	s_1(n) 
	= \frac{\delta(n) - \delta(n+1) + \floor{\log_2(n)}}{2}+1.
	$$
\end{corollary}

\section{From \texorpdfstring{$D$}--nodes to the Takagi function}\label{dnodes2takagi}
In this section, we take the identities on $D$-nodes on a divide-and-conquer tree from Section~\ref{pairwise_cnaos} and apply Theorem~\ref{takagi2} to obtain identities on the Takagi function.
\begin{theorem} 
	\label{th_yet_another_takagi_prop}
	Let $1 \le r\le 2^k-1$. Let $\rho_1$ be the position of the smallest non-$0$ bit in $r$. Then
	\begin{equation}\label{eq_yet_another_takagi_prop}
		\tau\left(\frac{r}{2^k}\right)=
		\frac{1}{2} \left(\tau\left(\frac{r-1}{2^k}\right) + \tau\left(\frac{r+1}{2^k}\right)\right) + \frac{1-\rho_1}{2^k}.
	\end{equation}
\end{theorem}
\begin{proof}
	Follows from Theorem~\ref{takagi2} and Corollary~\ref{yet_another_dnodes_prop}. 
\end{proof}
We observe that when $x=\dfrac{r-1}{2^k} \text{ and } y=\dfrac{r+1}{2^k},$  Theorem~\ref{th_yet_another_takagi_prop} gives rise to a special case of Boros' inequality \cite{boros08}:
$$
\tau\bigl(\frac{x+y}{2}\bigr) \le \frac{1}{2}(\tau(x)+\tau(y))+\frac{\bigl|x-y\bigr|}{2}.
$$
For this $x$ and $y$, Equation (\ref{eq_yet_another_takagi_prop}) in the theorem may be restated as follows:
$$
\tau\bigl(\frac{x+y}{2}\bigr)= \frac{1}{2}(\tau(x)+\tau(y))+\frac{\left|x-y\right|}{2} - \frac{\rho_1}{2^k},
$$ 
so here, equality holds in Boros' inequality if and only if $\rho_1=0$, i.e., $r$ is odd. 
\begin{theorem}
	\label{cor_takagi_exp_D_closed}
	Let $r$ be an integer. The Takagi function on $\frac{r}{2^k}$ satisfies 
	the equation
	$$
	\tau\left(\frac{r}{2^k}\right)=
	\frac{1}{2^k} \cdot \sum_{i=0}^{k-1} \lambda_i(n) \text{, where } \lambda_i(n)= 
	\begin{cases}
		(n\bmod 2^i), & \text{ if } n_i=0;\\
		2^i-(n\bmod 2^i), & \text{ if } n_i=1.
	\end{cases}	  
	$$ 
	An explicit form is
	$$
	\tau\left(\frac{r}{2^k}\right)=
	\frac{1}{2^k} \cdot \sum_{i=0}^{k-1} (n_i\cdot 2^i + (-1)^{n_i}\cdot (n \bmod 2^i)).
	$$
\end{theorem}
\begin{proof}
	Follows from Theorem~\ref{exp_D_closed}.
\end{proof}
\begin{theorem} 
	\label{takagi_another_recurrence} 
	Let $r$ be an integer. The Takagi function on $\frac{r+1}{2^k}$ 
	satisfies the equation
	$$
	\tau\left(\frac{r+1}{2^k}\right) 
	= \tau\left(\frac{r}{2^k}\right) + \frac{1}{2^k}\cdot (k-2\cdot s_1(n)+2).
	$$
\end{theorem}
\begin{proof}
	Follows from Theorem~\ref{dnodes_another_recurrence}.
\end{proof}
Theorem~\ref{takagi_another_recurrence} 
may also be derived from Kr\"uppel's Proposition 2.1
\cite[Formula (2.2)]{kruppel07}
by substituting $1$ for $x$, $r$ for $k$, and $k$ for $\ell$ in the statement of the formula.
\begin{theorem} \label{th_another_explicit_takagi}
	Let $r$ be an integer. The Takagi function on $\frac{r}{2^k}$ 
	satisfies the equation
	$$
	\tau\left(\frac{r}{2^k}\right)=
	\frac{1}{2^k} \cdot \sum_{i=0}^{k-1}2^i \cdot n_i \cdot \bigl(k-i-2\cdot \sum_{j=i}^k n_i+4\bigr),
	$$
	or alternatively
	$$
	\tau\left(\frac{r}{2^k}\right)=
	\frac{1}{2^k} \cdot \sum_{i=0}^{k-1}2^i \cdot n_i \cdot \bigl(k-i-2\cdot s_1\left(\floor{\frac{n}{2^i}}\right)+4\bigr).
	$$
\end{theorem}
\begin{proof}
	Follows from Theorem~\ref{another_explicit_D}.
\end{proof}
We end the paper with Theorem~\ref{th_hamming_takagi} on the Hamming weight of binary integers and its Corollary~\ref{cor_trollope} on the cumulative sum of the weights of integers. Corollary~\ref{cor_trollope} was shown by Trollope in 1968 \cite{trollope68}, but also follows directly from the weight result, so we provide that proof.
\begin{theorem}\label{th_hamming_takagi}
	The Hamming weight of $n$, $s_1(n)$, satisfies the equation
	$$
	s_1(n) 
	= 2^{k-1} \left(\tau\left(\frac{r}{2^k}\right) - \tau\left(\frac{r+1}{2^k}\right)\right) + \frac{k+2}{2}.
	$$
\end{theorem}
\begin{proof}
	Follows from Theorem \ref{takagi_another_recurrence}.
\end{proof}
\begin{lemma}\label{lemma_sumpow2}
	Let $S_1(n)=\sum_{i=0}^{n-1} s_1(i)$. If $n$ is a power of 2, then $S_1(n)=\frac{1}{2} \cdot n\cdot \log_2(n)$.
\end{lemma}
\begin{proof}
	Exactly half of the bits in the integers between $1$ and $n-1$ have value $1$.
\end{proof}
\begin{corollary} \cite{trollope68}\label{cor_trollope}
	Let $S_1(n)=\sum_{i=0}^{n-1} s_1(i)$. 
	Then
	$$S_1(n) = \frac{n \cdot \log_2(n)}{2} + 2^{k-1}\cdot (2x - \tau(x) - (1+x)\cdot \log_2(1+x)).$$
\end{corollary}
\begin{proof}
	\begin{align}
		S_1(n) &= \sum_{i=0}^{2^k-1} s_1(i) + \sum_{i=2^k}^{2^k+r-1} s_1(i) \nonumber\\
		&= \frac{2^k\cdot k}{2} + \sum_{i=2^k}^{2^k+r-1} s_1(i) &\text{(by Lemma~\ref{lemma_sumpow2})} \nonumber\\
		&= \frac{2^k\cdot k}{2} + 2^{k-1}\cdot \left(\tau\left(\frac{0}{2^k}\right)-\tau\left(\frac{r}{2^k}\right)\right)+ \sum_{i=2^k}^{2^k+r-1} \frac{k+2}{2} &\text{(by Theorem~\ref{th_hamming_takagi})} \nonumber\\
		&= \frac{2^k\cdot k}{2} - 2^{k-1}\cdot \tau\left(\frac{r}{2^k}\right)+ r \cdot \frac{k+2}{2}  \nonumber\\
		&= \frac{(2^k+r)\cdot k}{2} - 2^{k-1}\cdot \tau\left(\frac{r}{2^k}\right)+ r   \nonumber\\
		&= \frac{(2^k+r)\cdot k}{2} + 2^{k-1}\cdot \left(2\cdot \frac{r}{2^k} - \tau\left(\frac{r}{2^k}\right) \right)   \nonumber\\
		&= \frac{n\cdot k}{2} + 2^{k-1}\cdot(2x - \tau(x) )  &\text{(by Lemma \ref{lem_conversion})}\label{eq_sleek}\\
		&= \frac{n\cdot k}{2} + \frac{n}{2}\cdot \log_2\left(\frac{n}{2^k}\right) + 2^{k-1}\cdot \left(2x - \tau(x) - \frac{n}{2^k}\cdot \log_2\left(\frac{n}{2^k}\right) \right)  \nonumber\\[1ex]
		&= \frac{n\cdot k}{2} + \frac{n}{2}\cdot (\log_2n-k) + 2^{k-1}\cdot \left(2x - \tau(x) - \frac{n}{2^k}\cdot \log_2\left(\frac{n}{2^k}\right) \right)   \nonumber\\[1ex]
		&= \frac{n\cdot \log_2n}{2} + 2^{k-1}\cdot \left(2x - \tau(x) - \frac{n}{2^k}\cdot \log_2\left(\frac{n}{2^k}\right) \right)  \nonumber\\[1ex]
		&= \frac{n\cdot \log_2n}{2} +  2^{k-1}\cdot (2x - \tau(x) - (1+x)\cdot \log_2(1+x) ). &\text{(by Lemma \ref{lem_conversion})} \nonumber
	\end{align}
\end{proof}
Finally, Corollary~\ref{cor_trollope3way} gives three succinct formulas of the cumulative binary digit sum in terms of the Takagi function and in terms of the number of $D$-nodes.
\begin{corollary}\label{cor_trollope3way}
	Let $n$ be as in Notation~\ref{not_n}, and let $S_1(n)=\sum_{i=0}^{n-1} s_1(i)$. Then 
	\begin{center}
		\begin{minipage}{.5\textwidth}
			\begin{enumerate}
				\item $S_1(n) = \frac{1}{2}\cdot (nk + 2^k\cdot  (2x - \tau(x) ))$,\label{cor_trollope3way_1}
				\item $S_1(n)=\frac{1}{2}\cdot \left(nk + 2r  -2^k\cdot \tau\left(\frac{r}{2^k}\right)\right) $, and\label{cor_trollope3way_2}
				\item  $S_1(n) = \frac{1}{2}\cdot (nk +2 r -  \delta(n)) $.\label{cor_trollope3way_3}
			\end{enumerate}
		\end{minipage}
	\end{center}
\end{corollary}
\begin{proof}
	Follows from Equation (\ref{eq_sleek}) in the proof of Corollary~\ref{cor_trollope}, and Theorem~\ref{takagi2}.
\end{proof}
The first two subcorollaries of Corollary~\ref{cor_trollope3way} may also be derived from Kr\"uppel's Proposition 2.1 (Formula (2.4) \cite{kruppel07}, by substituting $r$ for $n$ and $\ell$ for $k$ in the statement of the formula and using Lemma~\ref{lem_conversion} to obtain $S_1(n)$. The third subcorollary  would then follow by Theorem~\ref{takagi2}.

\section{Acknowledgments}
We thank Jean-Paul Allouche for calling our attention to Kr\"uppel's results on Hamming weights. We also thank the anonymous reviewers for careful reading and helpful remarks.

This paper has been assigned the LANL identification number LA-UR-24-21017.
The U.S. Government retains an irrevocable, nonexclusive, royalty-free license to publish, translate, reproduce, use, or dispose of the published form of the work and to authorize others to do the same.

	\bibliographystyle{jis}
	\bibliography{biblio}

\begin{thebibliography}{10}

\bibitem{allaart11}
P.~Allaart and K.~Kawamura, The {T}akagi function: a survey, {\em Real Anal.
  Exchange} {\bf 37} (2011), 1--54.

\bibitem{baruchel19_2}
T.~Baruchel, Properties of the cumulated deficient binary digit sum, ar{X}iv
  preprint ar{X}iv:1908.02250 [math.{NT}], 2019.
\newblock Available at \url{https://arxiv.org/abs/1908.02250}.

\bibitem{boros08}
Z.~Boros, An inequality for the {T}akagi function, {\em Math. Inequal. Appl.}
  {\bf 11} (2008), 757--765.

\bibitem{coronado20}
T.~M. Coronado, M.~Fischer, Lina. Herbst, F.~Rossell\'o, and Kristina Wicke, On
  the minimum value of the {C}olless index and the bifurcating trees that
  achieve it, {\em J. Math. Biol.} {\bf 80} (2020), 1993--2054.

\bibitem{hwang17}
H.-K. Hwang, S.~Janson, and T.-H. Tsai, Exact and asymptotic solutions of a
  divide-and-conquer recurrence dividing at half: theory and applications, {\em
  ACM Trans. Algorithms} {\bf 13} (2017), 1--43.

\bibitem{knuth97art1}
D.~E. Knuth, {\em The Art of Computer Programming: Volume 1: Fundamental
  Algorithms}, Pearson Education, 1997.

\bibitem{kruppel07}
M.~Kr{\"u}ppel, On the extrema and the improper derivatives of {T}akagi's
  continuous nowhere differentiable function, {\em Rostock. Math. Kolloq.} {\bf
  62} (2007), 41--59.

\bibitem{lagarias11}
J.~Lagarias, The {T}akagi function and its properties, {\em RIMS
  K\^{o}ky\^{u}roku Bessatsu} {\bf B34} (2012), 153--189.

\bibitem{monroe21}
L.~Monroe, A class of trees having near-best balance, ar{X}iv preprint
  ar{X}iv:2108.11496 [cs.{DM}], 2021.
\newblock Available at \url{https://arxiv.org/abs/2108.11496}.

\bibitem{OEIS}
\relax OEIS Foundation~Inc., The {O}n-{L}ine {E}ncyclopedia of {I}nteger
  {S}equences, 2020.
\newblock References sequences A268289 and A296062. Available at
  \url{http://oeis.org}.

\bibitem{takagi01}
T.~Takagi, A simple example of the continuous function without derivative, {\em
  Phys.-Math. Soc. Japan} {\bf 1} (1903), 176--177.
\newblock [Collected Papers of Teiji Takagi (S. Iyanaga, Ed), Springer Verlag,
  New York 1990].

\bibitem{trollope68}
J.~R. Trollope, An explicit expression for binary digital sums, {\em Math.
  Mag.} {\bf 41} (1968), 21--25.

\end{thebibliography}

	\bigskip
	\hrule
	\bigskip
	
	\noindent 2020 {\it Mathematics Subject Classification}:
	Primary 68R05; Secondary 26A27, 05C05, 28A80.
	
	\noindent \emph{Keywords: } Takagi function, full binary tree, divide-and-conquer tree, Hamming weight.
	
	\bigskip
	\hrule
	\bigskip
	
	\noindent (Concerned with OEIS sequences
\seqnum{A000001},
\seqnum{A268289}, and
\seqnum{A296062}.)
	
	\bigskip
	\hrule
	\bigskip

\end{document}